\documentclass[12pt, a4paper, reqno, oneside]{amsart}
\usepackage{amssymb}  
\usepackage{amsmath} 
\usepackage{amsthm}  

\usepackage[cp1251]{inputenc}
\usepackage[english]{babel}
\usepackage{cite, verbatim}
\usepackage{multirow} 

\newtheorem{theorem}{Theorem}
\newtheorem{corollary}{Corollary}
\newtheorem{lemma}{Lemma}
\newtheorem*{question*}{question}

\newenvironment{nouppercase}{%
  \renewcommand{\uppercasenonmath}[1]{}}{}

\begin{document}

\title{\Large \rm On element orders in covers of $L_4(q)$ and $U_4(q)$ }

\author{M.A. Grechkoseeva, S.V. Skresanov}

\thanks{The reported study was funded by RFBR, project number 18-31-20011}

\begin{abstract}
	Suppose that $L$ is one of  the finite simple groups $\operatorname{PSL}_4(q)$ or $\operatorname{PSU}_4(q)$
	and $L$ acts on a vector space $W$ over a field whose characteristic divides $q$. We prove that the~natural semidirect product of 
	$W$ and $L$ contains an element whose order differs from the order of any element of $L$, thus answering 
	questions 14.60 and 17.73\,(a) of the~Kourovka Notebook.
	\smallskip

	\noindent
	{\bf Keywords:} simple linear group, simple unitary group, orders of elements, modular representation, defining characteristic.
	\smallskip

	\noindent
	{\bf MSC:} 20D06, 20D60, 20C20.
\end{abstract}
 
\begin{nouppercase}

\maketitle

\end{nouppercase}

\section{Introduction}

Problem 14.60 of the Kourovka Notebook \cite{Kou} 
asks whether a finite group $H$ having a~nontrivial normal subgroup $K$ such that 
the factor group $H/K$ is isomorphic to one of the simple groups $\operatorname{PSL}_n(q)$ with $n\geqslant 3$ always contains an element whose 
order is distinct from the order of any element of $H/K$. Zavarnitsine \cite{08Zav1.t} proved that this is true if $n\neq 4$
or $q$ is composite. Later he  \cite{08Zav2} constructed an example showing that this is not true if $n=4$ and $q=13^{24}$. 
It turned out that the proof in \cite{08Zav2} was incorrect, and the main goal of the present paper is to show that the answer is 
affirmative for all $n\geqslant 3$, including $n=4$. 

Before stating the main result, let us  put the question into a wider context of recognition by spectrum (see the introduction in \cite{15VasGr1} for a survey on this subject). The spectrum 
$\omega(G)$ of a finite group $G$ is the set of the orders of elements of $G$. We say that $G$ is recognizable by spectrum  if for every finite group $H$, the equality $\omega(H)=\omega(G)$ implies 
$H\simeq G$.  If the implication holds for all finite groups $H$ having a normal subgroup $K$ such that $H/K\simeq G$, 
then we say that $G$ is recognizable by spectrum among covers. In this language, Problem 14.60 asks whether the simple groups $\operatorname{PSL}_n(q)$ with $n\geqslant 3$ are recognizable by spectrum 
among covers. 

It is not hard to see that a finite group $G$ is recognizable among covers if and only if $\omega(G)\neq \omega(K:G)$, where $K\neq 1$ is an elementary abelian group  and $K:G$ is a split extension of 
$G$ by $K$ (see, for example, \cite[Lemma 12]{04Zav}), and so representations of $G$ come into play. If $G$ is a simple group of Lie type in characteristic $p$, it is natural to distinguish two 
cases, depending on whether $p$ divides $K$ or  not (cf. Problems 17.73 and 17.74 of \cite{Kou}). At the present time, the only open question related to recognizability of simple groups among covers 
is whether $\omega(G)\neq \omega(K:G)$ in the case when $G=\operatorname{PSL}_4(q)$ or $\operatorname{PSU}_4(q)$,
$q=p^m$ is odd and composite, and $K$ is an elementary abelian $p$-group (see the proof of Corollary \ref{c:1} 
below  for references). We answer this question in affirmative, thus solving Problems 14.60 and 17.73\,(a) of \cite{Kou}. 

\begin{theorem}\label{t:1}
	Let $L$ be $\operatorname{PSL}_4(q)$ or $\operatorname{PSU}_4(q)$, where $q$ is a power of an odd prime $p$.
	If $L$ acts on a vector space $W$ over a field of characteristic $p$ then $\omega(W\rtimes L)\neq\omega(L)$. 
\end{theorem}

\begin{corollary}\label{c:1}
	A finite nonabelian simple group $L$ is recognizable by spectrum among covers if and only if $L$ is neither  $^3D_4(2)$, nor
	$\operatorname{PSU}_5(2)$, nor $\operatorname{PSU}_3(q)$, where $q$ is a Mersenne prime such that $q^2-q+1$ is a prime.
\end{corollary}

\section{Proof of the main result}

We start with some definitions and preliminary results.  
 
If $a$ is a nonzero integer, then the highest power of $2$ dividing $a$ is called the $2$-part of $a$ and is denoted by $(a)_2$. The following lemma is an easy consequence of the definition.  

\begin{lemma}\label{l:2part}
Let $a$ and $b$ be nonzero integers.
\begin{enumerate}
\item If $(a)_2\geqslant (b)_2$, then $(a+b)_2\neq (a)_2$.
\item If $(a)_2=(b)_2$, then $(a+b)_2>(a)_2$.
\end{enumerate}
\end{lemma}

The next lemma was proved by Bang \cite{86Bang}. Also it is a special case of Zsigmondy's theorem \cite{Zs}.

\begin{lemma}\label{l:zs}
Suppose that $\varepsilon \in\{+,-\}$ and $a,n\geqslant 2$ are integers. Then either there is a prime $r$ such that 
$r$ divides $a^n-(\varepsilon1)^n$ and does not divide $a^i-(\varepsilon1)^i$ for all $1\leqslant i<n$, or one of the following holds:
\begin{enumerate}
 \item $\varepsilon=+$, $n=6$, $a=2$;
 \item $\varepsilon=-$, $n=3$, $a=2$;
 \item $n=2$ and $a+\varepsilon1$ is a power of $2$. 
\end{enumerate}
\end{lemma} 

We refer to the prime $r$ in Lemma \ref{l:zs} as a primitive prime divisor of $a^n-(\varepsilon1)^n$ and denote some primitive divisor, if any, by $r_n(\varepsilon a)$.

Let $F$ be the algebraic closure of a field of prime order $p$ and let $G=\operatorname{SL}_n(F)$. We will need some information about weights of rational finite dimensional $FG$-modules,
which will be called simply $G$-modules for brevity. All unexplained terminology can be found, for example, in \cite{03Jan}. We can choose the group $D$ of diagonal matrices in $G$ 
as a maximal torus of $G$. If $M$ is a $G$-module, then $\Omega(M)$ is the set of weights of $M$ (relative to $D$). The irreducible 
$G$-module with the highest weight $\lambda$ is denoted by $M(\lambda)$. If $V=F^n$ is the natural $G$-module with canonical basis $e_1,\dots,e_n$, then $e_i$ is a weight vector for each 
$i\in\{1,\dots,n\}$, and we denote the corresponding weight by $\varepsilon_i$.
The Frobenius map on $G$ is the map defined by $(a_{ij})\mapsto (a_{ij}^p)$. 
If $M$ is a $G$-module, then the composition of the corresponding representation and the $i$th power of Frobenius map is also a representation of $G$ 
on $M$, and we denote the corresponding module by $M^{(p^i)}$.

\begin{lemma}\label{l:micro}
Let $G=\operatorname{SL}_n(F)$ and let $M$ be an irreducible $G$-module with $p$-restricted highest weight. Then either $0\in\Omega(M)$ or there is a uniquely defined number 
$k\in\{1,\dots,n-1\}$ such that $\Omega(M)$ contains the set $\{\varepsilon_{i_1}+\varepsilon_{i_2}+\dots+\varepsilon_{i_k}\mid 1\leqslant i_1<i_2<\dots<i_k\leqslant n\}.$
\end{lemma}

\begin{proof}
See Lemmas 13 and 14 in \cite{08Zav1.t}. 
\end{proof}

We proceed now to prove Theorem \ref{t:1}. It is convenient to write $\operatorname{PSL}_4^\pm(q)$,
and related notation such as $\operatorname{SL}^\pm_4(q)$, to denote linear and unitary groups with $+$
corresponding to linear groups and $-$ to the unitary groups. 
So let $L=\operatorname{PSL}_4^\varepsilon(q)$, where $\varepsilon\in\{+,-\}$ and $q=p^m$. Let 
$S=\operatorname{SL}_4^\varepsilon(q)$ and let $G=\operatorname{SL}_4(F)$, where $F$ is the algebraic closure of a field of order $p$, as above. 
We may regard $W$ as an $S$-module and by \cite[Lemmas 10 and 11]{08Zav1.t}, 
we may assume that $S$ acts on $W$ faithfully and absolutely irreducibly. By Steinberg's theorem \cite[Theorem 43]{68SteLec}, 
we also may assume 
that $W$ is the restriction to $S$ of an irreducible $G$-module $M(\lambda)$ for some $q$-restricted weight $\lambda$. By Steinberg's tensor product theorem \cite[Theorem 41]{68SteLec}, we have 
$$M(\lambda)\simeq M(\lambda_0)\otimes M(\lambda_1)^{(p)}\otimes \dots\otimes M(\lambda_{m-1})^{(p^{m-1})},$$
where $\lambda_i$ are $p$-restricted.  
By Lemma \ref{l:micro}, for 
each $i\in\{0,1,\dots, m-1\}$, there is $k_i\in\{0,1,2,3\}$ such that $\Omega(M(\lambda))$ contains all the weights of the 
form $$\mu_0+p\mu_1+\dots+p^{m-1}\mu_{m-1},$$ where $\mu_i$ is the zero weight if $k_i=0$, and $\mu_i$ is a sum of distinct $k_i$ elements of $\{\varepsilon_1,\dots,\varepsilon_{4}\}$ 
if $k_i>0$.

To prove that $\omega(W\rtimes L)\neq\omega(L)$, it is sufficient to find a semisimple element $g\in S$ such that $\langle g\rangle \cap Z(S)=1$, $p|g|\not\in\omega(L)$ and 
for each $i\in\{0,1,\dots, m-1\}$ with $k_i>0$, there are $k_i$ distinct characteristic values $\theta_{i,1},\dots,\theta_{i,k_i}$ of $g$ such that 
\begin{equation}\label{e:1}\prod_{i:\,k_i>0}(\theta_{i,1}\dots\theta_{i,k_{i}})^{p^{i}}=1.\end{equation} Indeed, there is $x\in G$ such that $h=g^x$ is 
a diagonal matrix. By \eqref{e:1} and the preceding paragraph, there is $\mu\in\Omega(M(\lambda))$ such that $\mu(h)=1$. It follows that $h$ has a 
nontrivial fixed point in $W$, and so too does $g$. Hence $p|g|\in\omega(W\rtimes L)\setminus\omega(L)$.

By \cite[Lemma 4]{08Zav2}, if $|g|$ is one of the following numbers, then   $p|g|\not\in\omega(L)$: 
\begin{enumerate}
 \item $r_4(\varepsilon q)$, $r_3(\varepsilon q)$, $(q^2-1)_2$;
 \item $r_2(\varepsilon q)(q-\varepsilon1)_2$ if $3<q\equiv \varepsilon1\pmod 4$.
\end{enumerate}
Observe that all the above primitive divisors exist by Lemma \ref{l:zs}.

Let $\theta\in F^\times$ have order $r_4(\varepsilon q)$ and let $g\in S$ be an element  whose characteristic values are $\theta$, $\theta^{\varepsilon q}$,
$\theta^{q^2}$, $\theta^{\varepsilon q^3}$. It is clear that $\theta^{q^2+1}=1$, $|g|=r_4(\varepsilon q)$, and $\langle g\rangle \cap Z(S)=1$.
If $k_i\in\{0,2\}$ for all $i$ and $\theta_{i,1}=\theta$, $\theta_{i,2}=\theta^{q^2}$ for $i$ with $k_i=2$, then \eqref{e:1} holds.  

Let $\theta\in F^\times$ have order $r_3(\varepsilon q)$ and let $g\in S$ be an element  whose characteristic values are $\theta$, $\theta^{\varepsilon q}$,
$\theta^{q^2}$, $1$.  If $k_i\neq 2$ for any $i$ then taking $\theta_{1,1}=1$ for $i$ with $k_i=1$
and $\theta_{i,1}=\theta$, $\theta_{i,2}=\theta^{\varepsilon q}$, $\theta_{i,3}=\theta^{q^2}$ for $i$ with $k_i=3$ gives us the desired result.

Thus we may assume that both $2$ and $1$, or both $2$ and $3$, occur among the numbers $k_i$. In particular, $q>p$. 

Let $q\equiv -\varepsilon\pmod 4$ and choose $\theta\in F^\times$ of order $(q^2-1)_2$. Observe that $\theta^{q+\varepsilon}=-1$ since $(q^2-1)_2=2(q+\varepsilon)_2$.
Let $g\in S$ be an element whose characteristic values are $\theta$, $\theta^{\varepsilon q}$, $-1$, $1$. For all $i$ with $k_i>0$, we set $\theta_{i,1}=1$.
If $k_i=2$, then $\theta_{i,2}=-1$, and if $k_i=3$, then $\theta_{i,2}=\theta$, $\theta_{i,3}=\theta^{\varepsilon q}$. Then the left side of \eqref{e:1} is equal to $1$ or
$-1$. If it is equal to $-1$, then we replace $\theta_{i,1}=1$ by $\theta_{i,1}=-1$ for some $i$ with $k_i\in\{1,3\}$.

Let $q\equiv \varepsilon\pmod 4$ and let $\theta\in F^\times$ have order $t=r(q-\varepsilon)_2$, where $r=r_2(\varepsilon q)$. Take $g\in S$ to be an element 
whose characteristic values are $\theta^a$, $\theta^{\varepsilon aq}$, $\theta^{rb}$, $\theta^{-a(1+\varepsilon q)-rb}$, where $a$ and $b$ are integers and $a$ is coprime to $r$.
Since $\theta^{t/2}=-1$, the characteristic values of $g^{t/2}$ are $(-1)^a$, $(-1)^a$, $(-1)^b$, $(-1)^b$. So if $a$ and $b$ have opposite parity, then 
$|g|=t$ and $\langle g\rangle \cap Z(S)=1$.  

If $k_i=2$, then set $\theta_{i,1}=\theta^{a}$ and $\theta_{i,2}=\theta^{a\varepsilon q}$. If $k_i=1$, then set $\theta_{i,1}=\theta^{rb}$. If $k_i=3$, then
set $\theta_{i,1}=\theta^a$, $\theta_{i,2}=\theta^{\varepsilon aq}$, and $\theta_{i,3}=\theta^{-a(1+\varepsilon q)-rb}$. Then the corresponding factor in the left side of \eqref{e:1} 
is equal to $\theta^{a(1+\varepsilon q)p^i}$, $\theta^{rbp^i}$, or $\theta^{-rbp^i}$ respectively. It follows that the product in the left side of \eqref{e:1} is 
equal to $\theta^c$, where 
 $$c=a(1+\varepsilon q)(p^{i_1}+\dots+p^{i_j})+rb(\tau_{j+1}p^{i_{j+1}}+\dots+\tau_{l}p^{i_l})$$
for some $\tau_{j+1},\dots,\tau_l\in\{+,-\}$. Define $A=(1+\varepsilon q)(p^{i_1}+\dots+p^{i_j})$ and $B=\tau_{j+1}p^{i_{j+1}}+\dots+\tau_{l}p^{i_l}$. 
Observe that $A$ is a nonzero integer since $j>0$ and $i_1,\dots, i_j$ are different positive integers. By similar reasons, $B$ is also nonzero.

It is clear that $r$ divides $c$, and hence $\theta^c=1$ if and only if  
\begin{equation}\label{e:2} aA+rbB\equiv 0\pmod{(q-\varepsilon)_2}.\end{equation}
If $(A)_2<(B)_2$, then we set $b=1$ and take $a$ to be a solution of the congruence 
$$aA/(A)_2\equiv -rB/(A)_2\pmod {(q-\varepsilon)_2}$$
coprime to $r$ (we can choose such a solution because  $r$ does not divide $q-\varepsilon$).
Since  both $B/(A)_2$ and $(q-\varepsilon)_2$ are even, while $A/(A)_2$ is odd, the number $a$ is even, as required.  Similarly, if $(A)_2>(B)_2$, then we set $a=1$ and take $b$ to be a solution of 
the congruence $$rbB/(B)_2\equiv -A/(B)_{2}\pmod {(q-\varepsilon)_2}.$$ 

Let $(A)_2=(B)_2$. Suppose that for some $i$ with $k_i=2$, we replace  $\theta_{i,1}=\theta^{a}$, $\theta_{i,2}=\theta^{\varepsilon aq}$ by $\theta_{i,1}=\theta^{rb}$,  
$\theta_{i,2}=\theta^{-a(1+\varepsilon q)-rb}$. Then the corresponding factor in the left side of \eqref{e:1} changes from $\theta^{a(1+\varepsilon q)p^i}$ to $\theta^{-a(1+\varepsilon q)p^i}$,
and so $A$ is decreased by $2(1+\varepsilon q)p^i$, while $B$ is unchanged. Observe that $A-2(1+\varepsilon q)p^i$ is still nonzero and divisible by $r$. If $(A)_2\geqslant 
2(1+\varepsilon q)_2=4$, then $(A-2(1+\varepsilon q)p^i)_2\neq (A)_2$ by Lemma \ref{l:2part}, and we can proceed as in the case $(A)_2\neq (B)_2$. 

We are left with the case $(A)_2=(B)_2=2$. If for some $i$ with $k_i=1$, we replace $\theta_{i,1}=\theta^{rb}$ by $\theta_{i,1}=\theta^{-a(1+\varepsilon q)-rb}$,
this decreases $A$ and $B$ by $(1+\varepsilon q)p^i$ and $2rp^i$ respectively. Similarly, if there is $i$ with $k_i=3$, then we can increase  
$A$ and $B$ by the corresponding amounts replacing $\theta_{i,3}=\theta^{-a(1+\varepsilon q)-rb}$ by $\theta_{i,3}=\theta^{rb}$. By Lemma \ref{l:2part}, we have $(A\pm(1+\varepsilon q)p^i)_2>(A)_2$, 
and so the previous argument goes through. 
The proof of the theorem is complete.

\medskip

It remains to prove the corollary. If $L$ is neither $\operatorname{PSL}_4(q)$, $\operatorname{PSU}_3(q)$, $\operatorname{PSU}_4(q)$, $\operatorname{PSU}_5(2)$, nor $^3D_4(2)$, then $L$ is recognizable among covers by \cite[Corollary 1.1]{15Gr}. For $\operatorname{PSU}_3(q)$, 
$\operatorname{PSU}_5(2)$ and $^3D_4(2)$, the result is proved in \cite{06Zav.t}, \cite[Proposition 2]{11Gr} and \cite{13Maz.t} respectively. 
If $q$ is even then $\operatorname{PSL}_4(q)$ and $\operatorname{PSU}_4(q)$ are recognizable by 
spectrum \cite[Theorem 1]{08MazChe.t}. Thus we may assume that $L=\operatorname{PSL}_4(q)$ or $\operatorname{PSU}_4(q)$ with $q$ odd. As we mentioned in Introduction, it suffices to show that $\omega(L)\neq\omega(K\rtimes L)$ for 
every elementary abelian group $K\neq 1$ with $L$-action. If $p$ and $|K|$ are coprime, the result follows from  \cite[Lemma 11]{08Zav1.t} for linear groups and \cite[Theorem 1]{11Gr} for unitary 
groups. If $p$ divides the order of $K$,  we apply Theorem \ref{t:1}.

\bigskip

{\footnotesize
\noindent{\sl Mariya A. Grechkoseeva\\
	Sobolev Institute of Mathematics, 4 Acad. Koptyug avenue,\\
	Novosibirsk 630090, Russia\\
	e-mail: grechkoseeva@gmail.com
}
\bigskip

\noindent{\sl Saveliy V. Skresanov\\
	Sobolev Institute of Mathematics, 4 Acad. Koptyug avenue,\\
	Novosibirsk State University, 1 Pirogova Str.,\\
	Novosibirsk 630090, Russia\\
	e-mail: skresan@math.nsc.ru
}}


\begin{thebibliography}{9}

\bibitem{Kou}
\emph{Unsolved problems in group theory. The Kourovka notebook},
19th ed., Inst. of Mathematics, SO RAN, Novosibirsk, 2018.

\bibitem{08Zav1.t}
A.V. Zavarnitsine,
\emph{Properties of element orders in covers for $L_n(q)$ and $U_n(q)$},
Siberian Math. J., \textbf{49}:2 (2008), 246--256.

\bibitem{08Zav2}
A.V. Zavarnitsine,
\emph{Exceptional action of the simple groups ${L}_4(q)$ in the defining characteristic},
Siberian Electronic Mathematical Reports, \textbf{5} (2008), 65--74.

\bibitem{15VasGr1}
M.A. Grechkoseeva, A.V. Vasil'ev,
\emph{On the structure of finite groups isospectral to finite simple groups},
J. Group Theory, \textbf{18}:5 (2015), 741--759.

\bibitem{04Zav}
A.V. Zavarnitsine,
\emph{Recognition of the simple groups $L_3(q)$ by element orders},
J. Group Theory, \textbf{7}:1 (2004), 81--97.

\bibitem{86Bang}
A.S. Bang,
\emph{Taltheoretiske Unders{\o}gelser},
Tidsskrift Math., \textbf{4} (1886), 70--80, 130--137.

\bibitem{Zs}
K.~Zsigmondy,
\emph{Zur Theorie der Potenzreste},
Monatsh. Math. Phys., \textbf{3} (1892), 265--284.

\bibitem{03Jan}
J.C. Jantzen,
\emph{Representations of algebraic groups. 2nd ed.}, vol. 107,
Amer. Math. Soc., Providence, RI, 2003.

\bibitem{68SteLec}
R. Steinberg,
\emph{Lectures on Chevalley groups},
Yale University, New Haven, Conn., 1968.

\bibitem{15Gr}
M.A. Grechkoseeva,
\emph{On element orders in covers of finite simple groups of Lie type},
J. Algebra Appl., \textbf{14} (2015), 1550056 [16 pages].

\bibitem{06Zav.t}
A.V. Zavarnitsine,
\emph{Recognition of the simple groups $U_3(q)$ by element orders},
Algebra Logic, \textbf{45}:2 (2006), 106--116.

\bibitem{11Gr}
M.A. Grechkoseeva,
\emph{On element orders in covers of finite simple classical groups},
J. Algebra, \textbf{339} (2011), 304--319.

\bibitem{13Maz.t}
V.D. Mazurov,
\emph{Unrecognizability by spectrum for a finite simple group $^3D_4(2)$},
Algebra Logic, \textbf{52}:5 (2013), 400--403.

\bibitem{08MazChe.t}
V.D. Mazurov, G.Y. Chen,
\emph{Recognizability of the finite simple groups $L_4(2^m)$ and $U_4(2^m)$ by the spectrum},
Algebra Logic, \textbf{47}:1 (2008), 49--55.

\end{thebibliography}
\end{document}